\documentclass[12pt]{article}
\usepackage{amssymb, amscd, amsmath, latexsym, hyperref}
\newtheorem{thm}{Theorem}[section]
\newtheorem{cor}[thm]{Corollary}
\newtheorem{df}[thm]{Definition}

\newtheorem{prp}[thm]{Proposition}
\newtheorem{rmk}[thm]{Remark}
\newtheorem{exm}[thm]{Example}

\usepackage{xcolor}

\numberwithin{equation}{section}

\numberwithin{equation}{section}

\def\S+{\mathbf{SO(n, 1)^+}}
\def\O+{\mathbf{O(n, 1)}}

\def\s+{\mathcal{S}}

\def\Z{\mathbb{Z}}
\def\N{\mathbb{N}}
\def\C{\mathbb{C}}

\newcommand {\cat}{\mathbf}

\date{}

\begin{document}
\title{Commutator identities, Lie product identities, and Multiplicative Lie algebras}

\author{Vipul Kakkar$^1$, Ramji Lal$^2$, Ravindra Prasad Shukla$^2$, Swapnil Srivastava$^3$\\ $^1$Department of Mathematics, Central University of Rajasthan \\ $^2$ Department of Mathematics, University of Allahabad\\ $^3$ Department of Mathematics, Ewing Christian College\\ $^3$math.swapnil@gmail.com}

\maketitle

\begin{abstract}
The aim of  this paper is to give a formal meaning of universal commutator identities and to see as to how the structure of multiplicative Lie algebra on a group is a potential candidate to describe universal commutator identities. We introduce some higher Schur multipliers. We shall also introduce and study Steinberg and  Chevalley Multiplicative Lie algebras, the multiplicative Lie algebras  with commutator identities among the root vectors describing Lie product relations. Free multiplicative Lie algebra on Steinberg and Chevalley groups are discussed.
\end{abstract}



\section{Introduction}
One of the most basic problem in group theory is to study the universal  commutator identities which hold in all groups and on free groups all the commutator identities are  derivable from universal commutator identities. For universal Lie product identities in $\Z$-Lie algebras, we have a theorem of Magnus (\cite{mks}) which asserts that for a free group $G\ =\ F(X)$ on a set $X$,   the natural Lie ring structure on $\oplus\sum_{n\in \N}\gamma_n(G)/\gamma_{n+1}(G)$ induced by commutator operation is free $\Z$-Lie ring $L\Z (X)$ on the abelianizer $G_{ab}$ of $G$ and which is free on the set $Y\ =\ \{xL_{2}(G)\mid x\in X\}$ in bijective correspondence with $X$ such that the $L\Z (X)$-module $(L\Z (X))_{n}$ generated by the set of homogeneous Lie elements of degree $n$ is $\gamma_n(G)/\gamma_{n+1}(G)$.  This, perhaps,  prompted Ellis (\cite{ell2}) to think of an structure, multiplicative Lie algebra structure $\star$  on a group which may describe universal $n$-commutator identities on a group more or less as Magnus theorem describes universal Lie product identities. From the universal property of free multiplicative Lie algebra on groups,  we have unique multiplicative Lie algebra homomorphism $\eta$ from the free multiplicative Lie algebra $LF(G)$ on $G$ to $G$ with improper multiplicative Lie algebra structure which induces surjective homomorphism $\eta_{n}$ from the subgroup $P_{n}(G)$ of $LF(G)$ generated by the set $\{(((....(g_{1}\star g_{2})\star g_{3})\star\cdots\star g_{n-1})\star g_{n})\mid g_{i}\in G\subset LF(G)\}$ to $\gamma_n(G)$ given by $\eta_{n}((((....(g_{1}\star g_{2})\star g_{3})\star\cdots\star g_{n-1})\star g_{n}))\ =\ [[[.....[g_{1}, g_{2}], g_{3}],\cdots ,g_{n-1}],g_{n}]$. It is conjectured that $\eta_{n}$ is an isomorphism for all $n$ whenever $G$ is a free group. The group structure on $LF(G)$ and  the homomorphism $\eta_{n}$ are important objects of study. The functor $Ker~\eta_{n}$ on the category $\cat{GP}$ of groups appears to be an interesting functor in the theory of groups.  We shall discuss the Schur theory including Schur-Hopf formula in the category of multiplicative Lie algebras. The purpose of this article is to make yet an other attempt, apart from those in (\cite{dl}, \cite{ell2}, \cite{lu}, \cite{mn}), to describe, of course a little differently, the notion of universal commutator identities in particular $n$-commutator identities ($n=2,3,\cdots$) which describe all $n$-commutator identities in the sense that on any group these commutator identities hold and on free groups, these commutator identities generate all commutator identities.

Together with its application in commutator calculus, the structure theory of multiplicative Lie algebra has become  of its own intrinsic interest.  The nilpotency and  solvability criteria for  multiplicative Lie algebras were discussed by F. Point and P. Wantiez (\cite{pw}) in 1996. The homology theory of multiplicative Lie algebras, Schur multiplier in terms of homology,  multiplicative Steinberg Lie algebra of a ring, and in turn, a new Milnor $K$-group $K'_{2}(R)\approx K_{2}(R)\times HC_{1}(R)$ were introduced and studied by A. Bak et al \cite{bdi} in 2007. In \cite{lu}, Schur-Hopf formula for multiplicative Lie algebras has been established,  the Schur multiplier has been described as group of non-trivial relations, and also a version of Schur multiplier in terms of second cohomology, all agreeing on finite multiplicative Lie algebras has been discussed.  The notion of Lie exterior square as an analogue of non-abelian exterior square of groups describing universal identities has also been introduced again in \cite{lu}. In this paper, we try to describe  universal $n$-commutator relations by introducing a $n$-ary structure $\star_{n}$ on a group with the help of presentation theory of multiplicative Lie algebras introduced in \cite{lu}.  In \cite{bdi}, 2-dimensional homological interpretation of the length 2, 3 commutator identities  using non-abelian homology is described, where as in \cite{lu}, the second cohomological interpretation is given.  We also introduce the concept of Chevalley multiplicative Lie algebra over a unital ring and study them in light of the earlier discussions and developments. We describe the relationship between the Scur multiplier of the multiplicative Lie algebra associated to a Chevalley group with Schur multiplier of the Chevalley group.
		
\section{Universal Commutator Identities, Preliminaries in  Multiplicative Lie algebra}

In this section, we introduce formally the notion of a minimal set of universal commutator identities, recall some basics in multiplicative Lie algebras from \cite{ell2}, \cite{lu} and establish some more basic facts in the structure theory for their subsequent use.   Thus, for a  given $n$, we look for a minimal set $X_{n}$ and a minimal finite subset $K_{n}$ of the free group  $F(t_{n}(F(X_{n})))$ on $t_{n}(F(X_{n}))$, where $t_{n}$ is the simple bracket arrangement of weight $n$, and $F(X_{n})$ is the free group on $X_{n}$ such that given any group $G$, the kernel $C_{n}(G)$ of the homomorphism $\eta_{n}$ from $F(t_{n}(G))$ to $\gamma_n(G)$ given by $\eta_{n}(((....((g_{1}, g_{2}), g_{3}),........),g_{n}))\ =\ [[[....[[{g_{1}, g_{2}}], g_{3}],........],g_{n}]$ contains the normal subgroup $B_{n}(G)$ of $F(t_{n}(G))$ generated by the set $K_{n}(G)$ of words obtained in $F(t_{n}(G))$ by replacing the symbols of the words  in $K_{n}$ by arbitrary members of $G$. Further,  in addition we require that $C_{n}(G)\ =\ B_{n}(G)$ whenever $G$ is a free group. The group $P_{n}(G)\ =\ F(t_{n}(G))/B_{n}(G)$ ($P_{2}(G)\ =\ G\wedge G$) may replace the role  of $G\wedge G$ for $n$-commutator relations. The group $C_{n}(G)/B_{n}(G)$ is the group of $n$-commutator relations of $G$ modulo the group $B_{n}(G)$ of universal $n$-commutator relations of $G$ ($n$-Schur multiplier $M_{n}(G)$ of $G$).

For example, consider the case $n=2$. In this case,

\[X_{2} = \{x, y, z\}\]

and

\begin{align*}
K_{2} & = \{(x, x), (x, y)(y, x), (xy, z)(x, z)^{-1}(^{x}y, ^{x}z)^{-1}\} \\
      & \subseteq F(t_{2}(F(X_{2})))\ =\ F(F(X_{2})\times F(X_{2})).
\end{align*}

Evidently, $B_{2}(G)\ =\ <K_{2}(G)>$ is the group of universal $2$-commutator relations in $G$, $C_{2}(G)$ the group of all $2$-commutator relations in $G$, $P_{2}(G)\ =\ F(t_{2}(G))/B_{2}(G)$ is the non-abelian exterior square $G\wedge G$ of $G$ and $C_{2}(G)/B_{2}(G)\ =\ H(G)\approx\ H_{2}(G, \Z )$ the Schur multipiler of $G$ (\cite{lu}). For free group $G$, $C_{2}(G)\ =\ B_{2}(G)$, since the Schur multiplier of a free group is trivial, and in turn, we have natural isomorphism from $P_{2}(G)$ to $\gamma_{2}(G)\ =\ [G , G]$.

Next, consider the case $n=3$. Observe that $[G, G]\supseteq \gamma_{3}(G)\ =\ [[G, G], G]$ and we have a surjective homomorphism $\eta_{2}$ from $P_{2}(G)\ =\ G\wedge G$ to $\gamma_{2}(G)$ which is isomorphism whenever $G$ is free. Also $G$ acts on $G\wedge G$ through the action $g\mapsto ^{g}$, where $^{g}$ is the automorphism of $G\wedge G$ given by $^{g}(h\wedge k)\ =\ ^{g}h\wedge ^{g}k$. Further, any $G$-equivariant homomorphism (in particular $\eta_{2}$) $\phi$ from $P_{2}(G)\ =\ G\wedge G$ to $G$ defines a semi-multiplicative Lie algebra structure $\star$ on $G$ in the sense of Definition \ref{s3d1} if we put $g\star h\ =\ \phi (g\wedge h)$.  Consider the subgroup $S_{3}(G)$ of $G\wedge G$ generated by $\{[g,h]\wedge k\mid g, h, k\in G\}$. Since $\gamma_{3}(G)$ is generated by simple $3$-commutators,  $\eta_{2}^{-1}(\gamma_{3}(G))\ =\ S_{3}(G)H_{2}(G)$, and $\eta_{2}|_{S_{3}(G)H_{2}(G)}$ is a surjective homomorphism to $\gamma_{3}(G)$. In turn, it induces surjective homomorphism from  $S_{3}(G)$ to $\gamma_{3}(G)$ whose kernel is $S_{3}(G)\bigcap H_{2}(G)$, and $S_{3}(G)/ S_{3}(G)\bigcap H_{2}(G)\approx \gamma_{3}(G)$. Further,  we have a surjective homomorphism $\phi_{3}$ from $F(t_{3}(G))$ to $ S_{3}(G)$ given by $\phi_{3} ((g, h), k)\ =\ [g, h]\wedge k$ which in turn gives a surjective homomorphism $\eta_{3}\ =\ \eta_{2}\circ \phi_{3}$ from $F(t_{3}(G))$ to $\gamma_{3}(G)$. We have an other subgroup $\hat{S}_{3}(G)$ of $G\wedge G$ generated by 

\[\{(g\wedge h)(k\wedge h)(h\wedge kg)\mid g, h, k\in G\},\]
 
and a homomorphism $\psi_{3}$ from $F(t_{3}(G))$ to $\hat{S}_{3}(G)$ given by 

\[\psi_{3} ((g, h), k)\ =\ (g\wedge h)(k\wedge h)(h\wedge kg).\]

Since $[[g, h], k]\ =\ [g, h][k, h][h, kg]$, 

\[\eta_{2} \circ \psi_{3}\ =\ \eta_{2} \circ\phi_{3}\ =\ \eta_{3}.\] 

Now,

\begin{align*}
 \phi_{3}^{-1}(K_{2}(G)) & = \{((g, h), [g, h]), ((g,h), [k,l])((k,l), [g, h]),\\
                         & ((gh, k), l)^{-1} ((^{[^{g}h, ^{g}k]}g, ^{[^{g}h, ^{g}k]}k ), ^{[^{g}h, ^{g}k]}l)((^{g}h, ^{g}k),l)\mid g, h, k, l\in G\}
\end{align*}

is a set of $3$-commutator identities in $G$ which are derivable from the earlier universal $2$-commutator identities.
Further, a variant 

\begin{center}
$[[g, h], ^{h}k] [[h, k], ^{k}g][[k, g], ^{g}h]\ =\ 1 $
\end{center}

of Jacobi  identity for groups hold in $\gamma_{3}(G)$ and as such $\{ ([g, h]\wedge  ^{h}k) ([h, k]\wedge ^{k}g)([k, g]\wedge ^{g}h)\mid g, h, k\in G\}$ is contained in $Ker ~\eta_{2} |_{S_{3}(G)}\ =\ S_{3}(G)\bigcap H_{2}(G)$. In turn, it follows that $\{ ((g, h), ^{h}k) ((h, k), ^{k}g)((k, g), ^{g}h)\mid g, h, k\in G\}$ is contained in $Ker~ \eta_{3}\ =\ C_{3}(G)$. However, it does not belong to $\phi_{3}^{-1}(K_{2}(G))$ even if $G$ is a free group. This prompts us to enlarge $X_{2}$ to $X_{3}\ =\ \{x, y, z, u\}$ and $K_{2}$ to 

\begin{align*}
K_{3} & =  \{((x, y), [x, y]), ((x,y), [z,u])((z,u), [x, y]), \\
      & ((xy, z), u)^{-1} ((^{[^{x}y, ^{x}z]}x, ^{[^{x}y, ^{x}z]}z ), ^{[^{x}y, ^{x}z]}u)((^{x}y, ^{x}z),u)\mid x, y, z, u\in X_{3}\} \\
			& \bigcup \{((x, y), ^{y}z) ((y, z), ^{z}x)((z, x), ^{x}y)\mid x, y, z\in X_{3}\}.
\end{align*}

Here again, we can see as consequence to \cite{ell2} that  $K_{3}$ is the set of universal $3$-commutator relations in the sense that $B_{3}(G)\ =\ <K_{3}(G)>\subseteq C_{3}(G)$ and equality holds whenever $G$ is free group. This really means that the $3$-commutator  identities

\begin{enumerate}
	\item[(i)] $[[x, x], y]=1$,
	\item[(ii)] $[[xy, z], u]\ =\ [[^{[^{x}y, ^{x}z]}x, ^{[^{x}y, ^{x}z]}z ], ^{[^{x}y, ^{x}z]}u][[^{x}y, ^{x}z],u]$,
	\item[(iii)] $[[x, yz], u]\ =\ [[^{[x, y]}(^{y}x), ^{[x, y]}(^{y}z)], ^{[x, y]}u][[x, y], u]$,
	\item[(iv)] $[[x, y], ^{y}z] [[y, z], ^{z}x][[z, x], ^{x}y] $, together with the identity
	\item[(v)] $[^{z}[x, y], u]\ =\ [[^{z}x, ^{z}y], u]$
\end{enumerate}

form basic universal $3$-commutator identities.

More generally, look at the subgroup  

\[S_{n}(G) = < \{([....[[g_{1}, g_{2}],g_{3}]\cdots ,g_{n-1}]\wedge  g_{n})\mid g_{i}\in G\}>\] 

of $G\wedge G$, $n\geq 3$. Clearly, $\eta_{2}^{-1}(\gamma_n(G))\ =\ S_{n}(G)H_{2}(G)$, $Ker~\eta_{2}|_{S_{n}(G)}\ =\ S_{n}(G)\bigcap H_{2}(G)$. We have a surjective homomorphism $\phi_{n}$ from $F(t_{n}(G))$ to $S_{n}(G)$ given by \\$\phi_{n}((....((g_{1}, g_{2}),g_{3})\cdots ,g_{n-1}),  g_{n})\ =\ ([\cdots [[g_{1}, g_{2}],g_{3}]\cdots ,g_{n-1}]\wedge  g_{n})$. Clearly, $< \phi_{n}^{-1}(K_{3}(G)) >\subseteq C_{n}(G)$, and it is conjectured  that the equality holds for all $n\geq 3$, whenever $G$ is free group (\cite{ell1}). Thus, we can take $X_{n}$ to be $\{x_{1}, x_{2}, \cdots, x_{n+1}\}$ and $K_{n}$ can be obtained as before. 

The conjecture really asserts that all basic universal $n$-commutator identities ($n\geq 3$) are derivable from the 5-fundamental commutator identities. This prompts us to concentrate on $T_{3}(G)\ =\ S_{3}(G)H_{3}(G)/H_{3}(G)$ where $H_{3}(G)\ =\ <K_{3}(G)>$ or more conveniently, $(G\wedge G)/H_{3}(G)$. Any $G$-equivariant map $\hat{\eta}$ from  $(G\wedge G)/H_{3}(G)$ to $G$ defines a multiplicative Lie algebra structure $\star$ on $G$ by putting $g\star h\ =\ \hat{\eta} ((g\wedge h) H_{3}(G))$ (see Definition \ref{s2d1}). The structure of multiplicative Lie algebra is more convenient object for the study of commutator relations on groups, and it will formalize discussions. 

Next, we recall some basics in multiplicative Lie algebras from \cite{ell2}, \cite{lu} and establishes some more basic facts on the structure theory for their subsequent use. As usual the notation $[x, y]$ stands for the commutator $xyx^{-1}y^{-1}$, and $^{y}x$ for the conjugate $yxy^{-1}$ of $x$. $1$ will denote the identity of the group. 

\begin{df}\label{s2d1} (\cite{ell2})
A group $(G, \cdot )$ together with a function $\{~,~\}:G\times G \rightarrow G$ is called a multiplicative Lie algebra if for all $x, y, z\in G$,  the following hold
\begin{enumerate}
	\item[(1)] $\{x,x\}=1$,
	\item[(2)] $\{x,y\cdot z\}=\{x,y\} \cdot ^y\{x,z\}$,
	\item[(3)] $\{x\cdot y,z\}=^x\{y,z\} \cdot \{x,z\}$,
	\item[(4)] $\{\{x,y\},^yz\}\cdot\{\{y,z\},^zx\}\cdot\{\{z,x\},^y\}=1$,
	\item[(5)] $^z\{x,y\}=\{^zx, ^zy\}$.
\end{enumerate}
\end{df}

For our convenience, we will denote $\{x,y\}$ by $x\star y$. Therefore the above conditions are represented as

\begin{enumerate}
	 \item[(1)] $x\star x\ =\ 1$, 

   \item[(2)] $x\star (y\cdot z)\ =\ (x\star y)\cdot ^{y}(x\star z )$, 

   \item[(3)] $(x\cdot y)\star z\ =\ ^{x}(y\star z)\cdot (x\star z )$, 

   \item[(4)] $((x\star y)\star ^{y}z )\cdot ((y\star z)\star ^{z}x )\cdot ((z\star x)\star ^{x}y )\ =\ 1$, 
     
		\item[(5)] $ ^{z}(x\star y)\ =\ (^{z}x\star ^{z}y)$. 
\end{enumerate}

The function $\{~,~\}$ or $\star$ is called the multiplicative lie product. By the abuse of language we shall say that $\star$ is a multiplicative Lie algebra structure on the group $G$.  The group operation is either juxtaposition or it will be denoted by $\cdot$.  We shall  say that $G$ is a multiplicative Lie algebra (the group operation and the Lie multiplication will be understood to be there). Ellis \cite{ell2} (see also \cite{lu}) established the following proposition. 

\begin{prp}\label{s2p1}(\cite{ell2})  Let $(G, \cdot , \star )$ be a multiplicative Lie algebra. Then following further identities hold:

(i) $1\star x\ =\ 1\ =\ x\star 1$,

(ii) $(x\star y)\cdot (y\star x)\ =\ 1$,

(iii) $^{x\star y}(u\star v)\ =\ ^{[x, y]}(u\star v)$,

(iv) $[(x\star y), z]\ =\ ([x, y]\star z)$,

(v) $x^{-1}\star y\ =\ ^{x^{-1}}((x\star y)^{-1})$ and also $x\star y^{-1}\ =\ ^{y^{-1}}((x\star y)^{-1})$.
\end{prp}

\begin{prp}\label{prp}
 In case, $G$ is free or a simple group then 
\[(x\star y)(x\star z)( [z, x]\star y) = (x\star yz) = ((yz\star x))^{-1},\]
for all $x,y,z$ in $G$. More generally, if $H_{2}(G) = \{0\}$, then also the above identity holds for all $x,y,z $. 
\end{prp}

\begin{proof}
From \cite{mks}, it follows that under the hypothesis, the map $\phi_{3}$ from the subgroup  $T_{3}(G)$ of $G$ generated by $\{((x\star y)\star z\mid x, y, z\in G\}$ to $\gamma_{3}(G)$ given by $\phi_{3}((x\star y)\star z)\ =\ [[x, y],z]$ is an isomorphism and the fact that  $[x, y][x, z][ [z, x], y]\ =\ [x, yz]\ =\ [yz, x]^{-1}$. 
\end{proof}

\begin{rmk}
\begin{enumerate}
	\item It may be observed that Proposition \ref{s2p1} holds in any structure $(G, \cdot , \star )$ satisfying the condition 1, 2, 3, and 5 of Definition \ref{s2d1}.  The identity (4) of the definition is not needed to establish it.
\item We can replace the identity (ii) of Proposition \ref{s2p1} with either of the defining condition (2) or (3) of Definition \ref{s2d1}. 

\end{enumerate}
\end{rmk}

\begin{exm}
\begin{enumerate}
	\item[(i)] On any group $G$, we have the \textit{trivial multiplicative Lie algebra} structure   $\star$ given by $x\star y\ =\ 1$ for all $x, y\in G$. The commutator operation $[~, ~]$ also defines a multiplicative Lie algebra structure on $G$ which will be termed as the improper multiplicative Lie algebra structure on $G$. Other multiplicative Lie algebra structures are called proper multiplicative Lie algebra structures on $G$. There are groups, for example free groups and simple groups,  on which there are no proper multiplicative Lie algebra structures. They will be termed as Lie simple groups. 

\item[(ii)] Lie rings and  Lie algebras  are multiplicative Lie algebras. Indeed a multiplicative Lie algebra $G$ is $\Z$-Lie algebra if and only if the group part of $G$ is abelian.

\item[(iii)]  For any non-abelian group $G$, the standard semidirect product $G\rtimes Inn(G)$ is a proper multiplicative Lie algebra with the Lie product $\star$ given by $(g, i_{h})\star (k, i_{l})\ =\ ([gh, kl], I_{G})$, where $I_{G}$ denotes the identity element of $Inn(G)$. Note that this multiplicative Lie algebra does not satisfy the identity in Proposition \ref{prp}.

\item[(iv)] Consider the Hisenberg  group $H_{3}(\C )$ of unipotent $3\times 3$ matrices with entries in the field $\C$ of complex numbers. We have natural isomorphism $\psi$ from the commutator subgroup $[H_{3}(\C ), H_{3}(\C )]$ to the additive group $\C$ of complex numbers. Let $\phi$ be a non-identity automorphism of $\C$ (note that there are uncountably many automorphisms of $\C$ and all of them are non continuous except the identity and conjugation). Then $\eta\ =\ \psi^{-1}o\phi o\psi$ is an automorphism of $[H_{3}(\C ), H_{3}(\C )]$. Define an operation $\star$ on $H_{3} (\C )$ by putting $A\star B\ =\ \eta ([A, B]) $. It can be observed that $(H_{3}(\C ), \cdot , \star )$ is a proper multiplicative Lie algebra.   More generally, if $G$ is a nilpotent group of class 2, and $\phi$  a non-identity non-trivial endomorphism of $[G, G]$. Then we have a proper  multiplicative Lie algebra structure $\star$ on $G$ given by $a\star b\ =\ \phi [a, b]$, $a, b\in G$.
\end{enumerate}
\end{exm}

\begin{rmk}
Let $(G, \cdot, \star )$ be a multiplicative Lie algebra with a topology $T$ on $G$ such that $(G, T)$ is a topological group. Then $\star $ need not be continuous unlike in case of Lie Algebras, where Lie product is continuous (see the above example (iv)). This prompts us to have a genuine concept of topological multiplicative Lie algebra. There are several groups $G$ (for example simple groups, free groups and many more) with the property that given any multiplicative Lie algebra structure $\star$ and a topological group structure $T$ on $G$, $(G, \star , T)$ is a topological multiplicative Lie algebra. However, not all groups have this property. For example on the Heisenberg topological group $H_{3}(\C )$ of $3\times 3$ uni-upper triangular matrices with complex entries have uncountably many multiplicative Lie algebra structures which are not topological multiplicative Lie algebra. It may be an interesting problem to describe groups on which any topological group structure and multiplicative Lie algebra structure make them topological multiplicative Lie algebra. 
\end{rmk}


The basic notions and results in multiplicative Lie algebras such as subalgebras, ideals, quotient, homomorphisms, isomorphism theorems are as usual. For example, an ideal of $G$ is a normal subgroup $H$ such that $g\star h\in H$ for all $g\in G$ and $h\in H$. It can be easily checked that if $H$ and $K$ are ideals, then the group product $HK$ is an ideal, and the subgroup $H\star K$ of $G$ generated by $\{h\star k\mid h\in H,\ k\in K\}$ is also an ideal. In particular, $G\star G$ is an ideal. 
  
 It can also be seen that the group center $Z(G)$ of $G$, the Lie center $LZ(G) = \{x\in G\mid x\star y = 1 \text{ for all } y\in G\}$, the group commutator $[G, G]$, and Lie commutator $G\star G$ are ideals of the Lie algebra $G$. The following Proposition follows by induction if we use Definition \ref{s2d1}(2) and Proposition \ref{s2p1}(iv).

\begin{prp}\label{s2p2}(\cite{lu}) (i) Let $(G, \star )$ be a multiplicative Lie algebra which is perfect as multiplicative Lie algebra in the sense $G\star G\ =\ G$. Then the Lie center $LZ(G)$ of $G$ is the same as the group center $Z(G)$ of $G$ as a group, and the Lie commutator $G\star G$ is contained in the group commutator $[G, G]$.

(ii) Let $(G, \star )$ be a multiplicative Lie algebra which is perfect as a group  in the sense that $[G, G]\ =\ G$. Then the group center $Z(G)$ is contained in the Lie center $LZ(G)$, and the group commutator $[G, G]$ is contained the Lie commutator $G\star G$.
\end{prp}
The following proposition  is crucial in the theory of Schur multiplier of multiplicative Lie algebras:

\begin{prp}\label{s2p3}(\cite{lu})
Let $G$ be a multiplicative Lie algebra such that $G/(LZ(G)\bigcap Z(G))$ is finite. Then $[G, G](G\star G)$ is finite.
\end{prp}

Throughout $\cat{ML}$ will denote the category of multiplicative Lie algebras, $\cat{LZ}$ will denote the category of $\Z$-Lie algebras, $\cat{ML_{tri}}$ will denote the category of trivial multiplicative Lie algebras, and $\cat{ML_{imp}}$ will denote the category of improper multiplicative Lie algebras. Clearly, the category $\cat{ML_{tri}}$ and $\cat{ML_{imp}}$ are isomorphic to the category $\cat{GP}$ of groups. All these categories are subcategories of $\cat{ML}$, and we have the inclusion functors. We have adjoints to these functors. The group abelianization functor $G\mapsto G/[G, G]$ is adjoint to the inclusion functor from $\cat{LZ}$ to $\cat{ML}$, $G\mapsto G/(G\star G)$ is adjoint to the inclusion functor from $\cat{ML_{tri}}$ to $\cat{ML}$. The following proposition describes the adjoint to the inclusion functor (forgetful functor) from $\cat{ML_{imp}}$ to $\cat{ML}$. 

\begin{prp}\label{s2p4}
Let $G$ be a multiplicative Lie algebra. Let $A(G)$ denote the ideal of $G$ generated by $\{(x\star y)^{-1}[x, y]\mid x, y\in G\}$. Then $G/A(G)$ is improper multiplicative Lie algebra such that for any Lie algebra homomorphism $f$ from $G$ to an improper multiplicative Lie algebra $K$, there is a unique homomorphism $\overline{f}$ from $G/A(G)$ to $K$ such that $\overline{f}o\nu\ =\ f$. Indeed, $A(G)$ is the smallest ideal  of $G$ by which if we factor, we get an improper multiplicative Lie algebra.
\end{prp}

Observe further that $A(G)\subseteq C_{G}(G\star G)\bigcap (G\star G)[G, G]$ for $G/(C_{G}(G\star G)\bigcap (G\star G)[G, G])$ is improper multiplicative Lie algebra.

The adjoint to the forgetful functor from $\cat{ML}$ to $\cat{GP}$ and from $\cat{ML}$ to the category $\cat{SET}$ of sets will be discussed in the next section.

\section{Some universal constructions, Free Multiplicative Lie Algebras, and Presentations}
We start with a more general structure.

\begin{df}\label{s3d1}
A group $G$ together with a binary operation $\star $ on $G$ will be termed as a pseudo-multiplicative Lie algebra if the identities (2), (3) and (5) of Definition \ref{s2d1} hold. More explicitly, a group $G$ with an operation $\star$ on $G$ will be termed as a pseudo-multiplicative Lie algebra if 

\begin{enumerate}
	\item[(i)] $x\star (y\cdot z)\ =\ (x\star y)\cdot ^{y}(x\star z)$,
	\item[(ii)] $(x\cdot y)\star  z\ =\ ^{x}(y\star z)\cdot (x\star z)$, and
	\item[(iii)] $^{x}(y\star z )\ =\ ^{x}y\star ^{x}z$
for all $x, y, z\in G$.
\end{enumerate}

If in addition
\begin{enumerate}
	\item[(iv)] $x\star x\ =\ 1$ for all $x\in G$,
\end{enumerate}

we will term it as  semi-multiplicative Lie algebra.
\end{df}

Evidently, every multiplicative Lie algebra is a pseudo multiplicative Lie algebra and also a semi-multiplicative Lie algebra. However, a semi-multiplicative Lie algebra need not be a multiplicative Lie algebra.  A pseudo-multiplicative Lie algebra on an abelian group is precisely a ring (not necessarily associative) structure on the group and it is a semi-multiplicative Lie algebra if in addition to it $x^{2}\ =\ 0$ for all $x$. However, it can be easily observed that  an associative  ring $R$ is a multiplicative Lie algebra if and only if (i) $x^{2}\ =\ 0$ and (ii) $3xyz\ =\ 0$ for all $x, y, z\in R$. A non-associative ring is multiplicative Lie algebra if and only if it is a Lie ring.

\begin{rmk}
 From the proof of the respective identities in Proposition \ref{s2p1}, it follows that it still holds for semi-multiplicative Lie algebra. The identity 4 in the definition of multiplicative algebra is not used in establishing the identities. Furthermore, all the identities in Proposition \ref{s2p1} hold without the identities 1 and 4 except (ii).
\end{rmk}

We have a category $\cat{SMult}$ of semi-multiplicative Lie algebras with obvious morphisms. The forgetful functors from $\cat{SMult}$ to $\cat{SET}$ and $\cat{GP}$ have left adjoint to be described below.

Let $X$ be a set and $F(\beta (X))$ be the free group on the set $\beta (X)\ =\ \bigcup_{n\in \N}\beta_{n}$ of all bracket arrangements of members of  $X$, where $\beta_{n}(X)$ denotes the set of bracket arrangements of weight $n$ of members of $X$. We have an obvious operation $\star$ on $\beta(X)$ defined by $x\star y\ =\ \beta_2(x, y)$ for all $x, y\in \beta (X)$. Forced by the defining conditions (i) and (ii) of semi-multiplicative Lie algebra, using induction on the lengths of freely reduced words in $F(\beta (X))$, we extend this operation to an operation  $\hat{\star}$ on $F(\beta (X))$ which satisfies the identities (i) and (ii) of Definition \ref{s3d1}. First, we define $1\hat{\star }w \ =\ 1\ =\ w\hat{\star } 1$ for all $w\in F(\beta (X))$.  Next, we define $u\hat{\star}w$ for each reduced word $w$ in $F(\beta (X))$, where $u$ is a reduced word in $F(\beta (X))$. This we do by the induction on the length $l(u)$ of $u$. Suppose that $l(u)\ =\ 1$. Then $u\in \beta (X)$ or else $u^{-1}\in \beta (X)$. Suppose that $u\in \beta (X)$. Again, we use induction on $l(w)$ to define $u\hat{\star }w$. If $l(w)\ =\ 1$, then $w\in \beta (X)$ or else $w^{-1}\in \beta (X)$. If $w\in \beta (X)$, then already $u\hat{\star} w\ =\ (u, w)$. Suppose that  $w^{-1}\in \beta (X)$. Then the identities

\[1\ =\ u\hat{\star} (w^{-1}w)\ =\ (u\hat{\star}w^{-1})^{w^{-1}}(u\hat{\star} w)\]

\noindent forces us to put $u\hat{\star} w\ =\ w(u\hat{\star}w^{-1})^{-1}w^{-1}$ or equivalently, $w^{-1}(u\hat{\star} w)w\ =\ u\hat{\star}w^{-1}$. Suppose that $u\hat{\star}w$ has already been defined for all words $w$ of length $n$. Let $v\ =\ wx$ be a word of length $n+1$, where $w$ is a word of length $n$ and $x$ is a word of length 1. Then put $u\hat{\star}v\ =\ (u\hat{\star}w)(w(u\hat{\star}x)w^{-1})$. This defines $u\hat{\star{w}}$ for all $w\in F(\beta (X))$. Next if $u^{-1}\in \beta (X)$, then put $u\hat{\star}w\ =\ u(u^{-1}\hat{\star}w)^{-1}u^{-1}$. Thus, $u\hat{\star}w$ has been defined for all words $u$ of length 1 and all members $w$ of $F(\beta (X))$. Finally, suppose that $u\hat{\star}w$ has already been defined for words $u$ of length $n$ and all $w\in F(\beta (X))$. Let $ux$ be a word of length $n+1$. The condition (ii) forces us to put $(ux)\hat{\star}w\ =\ (u(x\hat{\star}w)u^{-1})(u\hat{\star}w)$. This defines the operation $\hat{\star}$ on $F(\beta (X))$.  From the construction, it follows that $(F(\beta (X)), \hat{\star} )$ is a structure satisfying (i) and (ii) of Definition \ref{s3d1}. 

Let $S$ be a subset of $F(\beta (X))$ describing certain identities. We describe the ideal of the structure $(F(\beta (X)), \hat{\star} )$ generated by $S$. Define $\Gamma^{n}(S)$ inductively as follows: Put $\Gamma^{1}(S)\ =\ \Gamma (S)$ to be the subgroup of $F(\beta (X))$ generated  by $\{wuw^{-1}\mid w\in F(\beta (X))\ and\ u\in S\bigcup F(\beta (X))\hat{\star}S\}$. Assuming that $\Gamma^{n}(S)$ has already been defined, define $\Gamma^{n+1}(S)\ =\ \Gamma (\Gamma^{n}(S))$. Clearly, the ideal $< S>$ generated by $S$ is $\bigcup_{n\in \N}\Gamma^{n}(S)$. It follows that $F(\beta (X))/< S> $ is a free structure on $X$ for a structure $\star $ on a group described by the identities (i), (ii) in $S$ and the identities given by $S$. In particular, $F(\beta (X))/ <S>$  denoted by $SLF(X)$ is a free semi-multiplicative Lie algebra on $X$, where

\begin{align*}
 S & = \{(u\hat{\star}u), (u\hat{\star}(vw))^{v}((u\hat{\star}w)^{-1})(u\hat{\star}v)^{-1}, (uv\hat{\star}w)(u\hat{\star}w)^{-1}(^{u}((v\hat{\star}w)^{-1})), ^{w}(u\hat{\star}v)^{-1}(^{w}u\hat{\star}^{w}v) \\
   & \mid u, v, w\in F(\beta (X))\}.
\end{align*}

Thus, the functor $SLF$ from the category $\cat{SET}$ of sets to the category $\cat{SML}$ of semi-multiplicative Lie algebras is adjoint to the forgetful functor from $\cat{SML}$ to $\cat{SET}$. 

Next, given a semi-multiplicative Lie algebra $(G, \star)$, the ideal $A$ of $(G, \star )$ generated by the subset

\[\{((x\star y)\star ^{y}z )\cdot ((y\star z)\star ^{z}x )\cdot ((z\star x)\star ^{x}y )\}\]

\noindent of $G$ is the smallest ideal of $G$ such that $G/A$ is multiplicative Lie algebra and the association $G\mapsto G/A$ is a functor which is adjoint to the forgetful functor from $\cat{Mult}$ to $\cat{SMult}$. In particular, $F(\beta (X))/ <U>$ denoted by $LF(X)$ is free multiplicative Lie algebra on $X$, where 

\begin{center}
$U\ =\ S\bigcup \{((x\star y)\star ^{y}z )\cdot ((y\star z)\star ^{z}x )\cdot ((z\star x)\star ^{x}y )\}$.
\end{center}

The functor $LF$ is an adjoint to the forgetful functor from $\cat{ML}$ to $\cat{SET}$.

Next, let $G$ be a group. Treating $G$ as a set, consider the ideal $< R >$ of $F(\beta (G))$  generated by the set $R\ =\ S\bigcup \{xy(xy)^{-1}\mid x, y\in G\}$. The semi-multiplicative Lie algebra $\hat{S}(G)\ =\  F(\beta (G))/<R>$ together with the natural group homomorphism $i$ from $G$ to $\hat{S}(G)$ given by $i(g)\ =\ g< R >$ is a free semi-multiplicative Lie algbra on the group $G$ in the sense that if $f$ is a group homomorphism from $G$ to a semi-multiplicative Lie algebra $H$, then there is a unique Lie homomorphism $\overline{f}$ from $\hat{S}{G}$ to $H$ such that $ \overline{f}\circ i\ =\ f$. Consequently, the association $G\mapsto \hat{S}{G}$ is a functor $\hat{S}$ from $\cat{SMULT}$ to $\cat{GP}$ which is an adjoint to the forgetful functor from $\cat{SMULT}$ to $\cat{GP}$. 

Similarly, $ LF (G)\ =\ F(\beta (G))/ <\hat{R}>$ is free multiplicative Lie algebra on the group $G$, where $\hat{R}\ =\ U \bigcup \{xy(xy)^{-1}\mid x, y\in G\}$, and the functor $LF$ is an adjoint to the forgetful functor from $\cat{ML}$ to $\cat{GP}$.

Now, we introduce the notion of presentation of a semi-multiplicative Lie algebra and also of multiplicative Lie algebras as usual. Let $(G, \star)$ be a semi-multiplicative Lie algebra (respectively multiplicative Lie algebra ). A pair $(X, R)$ together with a map $\phi$ from $X$ to $G$, where $R$ is a subset of $F(\beta (X))$ is called a presentation of $(G, \star)$ if the kernel of the unique Lie algebra homomorphism $\overline{\phi}$ from $SLF(X)$ to $G$ is $< R\bigcup S>/< S>$ (respectively $<R\bigcup U >/< U >$). More explicitly, a short exact sequence

\begin{center}
$1\longrightarrow H\stackrel{i}{\rightarrow}SLF(X)\stackrel{\alpha}{\rightarrow}G\longrightarrow 1$
\end{center}

\noindent may be termed as a presentation of $(G, \star)$. In particular, we have the canonical multiplication table presentation of $(G, R )$ of the semi-multiplicative Lie algebra $(G, \star )$, where $R\ =\ \{(xy)^{-1}xy\mid x, y\in G\}\bigcup \{(x, y)^{-1}(x\star y)\mid x, y\in G\}$.

\begin{rmk}
Given a set $X$, the Lie ring $\frac{LF(X)}{[LF(X), LF(X)]}$ is a free Lie ring on the set $X$. More explicitly, $\frac{LF(X)}{[LF(X), LF(X)]}$ is a free object in the category $\cat{LR}$ of Lie rings. Thus, under the abelianizer functor from $\cat{ML}$ to $\cat{LR}$, the free objects correspond.
\end{rmk}

Observe that the category of $\Z$-Lie algebra is a full subcategory of the category $\cat{ML}$ of multiplicative Lie algebras. The following proposition asserts that the free $\Z$-Lie algebra on abelian group is free multiplicative Lie algebra on the group.

\begin{prp}\label{s3p1}
Let $A$ be an abelian group. Then the group part of  free semi-multiplicative Lie algebra $SLF(A)$ on the group $A$ is an abelian group.  
\end{prp}

\begin{proof}
From (5), it follows that the elements of $A$ commutes with $x\star y$ for all $x, y\in A$, and from Proposition \ref{s2p1}(iii) it follows  that $x\star y$ commutes with $u\star v$ for all $x, y, u, v\in A$ Inductively, it follows that $LF(A)$ is commutative as a group.
\end{proof}

Since a multiplicative Lie algebra structure  on an abelian group is precisely a $\Z$-Lie algebra, the following corollary is immediate from the above Proposition. 

\begin{cor}\label{s3p1c1}
The free multiplicative Lie algebra on an abelian group is precisely free $\Z -$ Lie algebra on $A$.
\end{cor}

Let $G$ be a group. Let $\star$ and $\hat{\star}$ denote the Lie product in $LF(G)$ and $SLF(G)$ respectively. For any bracket arrangement in $\beta_{n}$ of weight $n$, we define subset $\hat{\beta}_{n}(G)$ of $SLF(G)$ and subset $\beta_{n}(G)$ of $LF(G)$ by induction as follows: We put $\hat{\beta}_{0}(G)\ =\ \{e\}\ =\ \beta_{0}(G)$, $\hat{\beta}_{1}(G)\ =\ G\ =\ \{g\equiv\ g<S >\in SLF(G)\}$, $\beta_{1}(G)\ =\ G\ =\ \{g\equiv\ g<U >\in LF(G)\}$,  $\hat{\beta}_{2}(G)\ =\ \{g\hat{\star} h\mid g, h\in G\}$, $\beta_{2}(G)\ =\ \{g\star h\mid g, h\in G\}$. Suppose that $\hat{\beta}_{r}(G)$ and $\beta_{r}(G)$ have been defined for all $r\leq n, n\geq 3$. Then $\beta_{n}\ =\ \beta_2(\beta_{r}, \beta_{s})$ for some bracket arrangements $\beta_{r}$, and $\beta_{s}$ of lengths $r$ and $s$ respectively, $r, s < n$. We define $\hat{\beta}_{n}(G)\ =\ \{x\hat{\star} y\mid x\in \hat{\beta}_{r}(G)\ and\ y\in \hat{\beta}_{s}(G)\}$.  Let $\hat{T}_{n}(G)$ denote the subgroup of $SLF(G)$ generated by the set $\hat{t}_{n}(G)$, and $T_{n}(G)$ the subgroup of $LF(G)$ generated by the set $t_{n}(G)$,   where $t_{n}$ denotes the simple bracket arrangement of weight $n$. More explicitly, $\hat{T}_{n}(G)\ =\ < \{((\cdots (x_{1}\hat{\star} x_{2})\hat{\star} x_{3})\cdots \hat{\star} x_{n-1})\hat{\star} x_{n})\mid x_{i}\in G\}>$, and $T_{n}(G)\ =\ < \{((\cdots (x_{1}\star x_{2})\star x_{3})\cdots \star x_{n-1})\star x_{n})\mid x_{i}\in G\}>$. In particular $G\approx \hat{T}_{1}(G)\subset SLF(G)\approx T_{1}(G)\subset LF(G)$. Thus, for example $\hat{T}_{1}(G)\ =\ T_{1}(G)= \ G$,  $T_{2}(G)\approx\ F(t_{2}(G))/B_{2}(G)\ \approx\ G\wedge G\approx \hat{T}_{2}(G)$, and $T_{3}(G)\approx\ S_{3}(G)/(S_{3}(G)\bigcap H_{2}(G))$.  The identity map $I_{G}$ on $G$ is a homomorphism from $G$ to $G$ treated as improper multiplicative Lie algebra. From the universal property of $SLF(G)$ and that of $LF(G)$, we get a unique surjective multiplicative Lie algebra homomorphism $\hat{\eta} $ from $SLF(G)$ to $G$ and $\eta$ from $LF(G)$ to $G$ which take $a\star b$ to $[a, b]$ or more generally, it maps $\hat{T}_{n}(G)$ and $T_{n}(G)$ to $\gamma_n(G)$ and the restriction maps are surjective homomorphisms. We denote the restriction map by $\eta_{n}$. Trivially, $\eta_{1}$ is a group isomorphism. In \cite{ell2}, $\eta_{2}$ and $\eta_{3}$ are shown to be isomomorphism, and in \cite{dl}, $\eta_{4}$ is also shown to be an isomorphism under the restriction that $G$ is free group. It justifies that the relations defining multiplicative Lie algebra is a set of universal defining relations for $n$ length commutator relations ($n=2, 3, 4$). In \cite{ell2}, it is also conjectured that $\eta_{n}$ is an isomorphism for each $n$ provided that $G$ is free. We look at the problem from different angle and try to find universal $n$-commutator relations and relate it with the above approach. In turn, we introduce the notion of $n$-Schur multiplier $M_{n}(G)$ with $M_{2}(G)$ the usual Schur multiplier for finite groups. 

\begin{prp}\label{s3p2}
$T_{n}(G)$ is a normal subgroup of $LF(G)$ and $\hat{T}_{n}(G)$ is a normal subgroup of $SLF(G)$ for all $n\geq 2$. However $T_{1}(G)= G$ need not be normal in $LF(G)$. 
\end{prp}

\begin{proof} 
Evidently, $LF{G}$($SLF(G$) is generated by the set $\bigcup_{\beta\in \Omega}\beta (G)$, where $\Omega$ is the set of all bracket arrangements.  The proof is  by the induction. For $n\ =\ 2$, $T_{2}(G)$ is generated by $\{g\star h\mid g, h\in G\}$. We again use induction to show that $T_{2}(G)$ is normal in $SLF(G)$.  From (5) it follows that $x(g\star h )x^{-1}\ =\ ^{x}g\star ^{x}h\in T_{2}(G)$ for each $x\in G\ =\ \beta_{1}(G)$. Assume that $uvu^{-1}\in T_{2}(G)$ for all $u\in \beta_{r}(G)$  and $v\in T_{2}(G)$ for all $r < n$, $n\geq 2$. It suffices to show that $wvw^{-1}\in T_{2}(G)$ for all $w\in \beta_{n}(G)$ and $v\ =\ g\star h,\ g, h\in G$. Now, $w\ =\ x\star y$ for some $x\in \beta_{r}(G)$ and $y\in \beta_{s}(G)$, where $r, s < n $ with $ r+s\ =\ n$. From Proposition \ref{s2p1}(iii), $wvw^{-1}\ =\ ^{x\star y}(g\star h)\ =\ ^{[x, y]}(g\star h)$ belongs to $T_{2}(G)$  for $[x, y]$ is product of elements of $\beta_{t}(G)$ with $t\leq max \{r, s\}$. Assume that $T_{r}(G)$ is normal for all $2\leq r\ <\ n,\ n\geq 3$. Using the similar argument as before, we can show that $T_{n}(G)$ is also normal. The proof is complete by the induction. The rest follows by same arguments. 
\end{proof}

\begin{prp}\label{s3p3}
$T_{n}(G)\supseteq \beta_{n}(G)$ for all bracket arrangement $\beta_{n}$ of weight $n$.
\end{prp}

\begin{proof}
We prove it by the induction on $n$. Trivially, $\beta_{1}(G)\ =\ \{(g)\mid g\in G\}\ =\ T_{1}(G)$ and also $\beta_{2}(G)\ \subseteq <\beta_{2} (G)>\ =\ T_{2}(G)$. Again, since $(g_{1}\star g_{2})\star g_{3}\ =\ (g_{3}\star (g_{1}\star g_{2}))^{-1}$, it also follows that $\beta_{3}(G)\subseteq T_{3}(G)$, where $\beta_{3}$ is any one of the bracket arrangement of weight 3. Assume the induction hypothesis. More explicitly, assume that $\beta_{r}(G)\subseteq T_{r}(G)$ for all $r < n$, $n\geq 4$. We prove it for $n$. Let $x$ be an element of $\beta_{n}(G)$.  Then  $x\ =\ \beta_{n} (g_{1}, g_{2}, \cdots , g_{n}), g_{i}\in G$, where  $\beta_{n}\ =\ (\gamma , \delta )$, $\gamma$  a bracket arrangement of weight $r$ and $\delta$ that of weight $s$ such that $r+s\ =\ n$ with $r, s\ < n$. Thus, $x\ =\ y\star z$, where $y\ =\  \gamma (g_{1},g_{2},\cdots , g_{r})\in \beta_{r}(G)$, and $z\ =\  \delta(g_{r+1}, g_{r+2}, \cdots , g_{n})\in \beta_{s}(G)$. By the induction hypothesis, $y\in T_{r}(G)$ and $z\in T_{s}(G)$. Suppose that $y\ =\ y_{1}^{\epsilon_{1}}y_{2}^{\epsilon_{2}}\cdots y_{p}^{\epsilon_{p}}$ and   $z\ =\ z_{1}^{\epsilon_{1}}z_{2}^{\epsilon_{2}}\cdots z_{q}^{\epsilon_{q}}$, where $y_{i}$ $(i\leq p)$ are in the generating set of $T_{r}(G)$ and $z_{i}$ ($j\leq q$) are in the generating set of $T_{s}(G)$, $\epsilon_{i}\ =\ \pm 1\ =\ \epsilon_{j}$. Using the defining conditions (2), (3), and (5) of the multiplicative Lie algebra, and proceeding inductively it can be easily observed that $x$ can be expressed as group product of elements of the types $v\star w$ or $w\star v$, where $v$ is a generating element of $ T_{r}(G)$ and $w$ is that of $ T_{s}(G)$ (note that $(v\star w)^{-1}\ =\ w\star v$). It suffices to show that $v\star w\in T_{n}(G)$. If $s\ =\ 1$, $v\star w$ is evidently a generating element of $T_{n}(G)$. Suppose that $s\geq 2$.  Then $w\ =\ u\star g$, where $u$ is a generating element of $ T_{s-1}(G)$ and $g\in G$. Consequently,
\begin{center}
$v\star w\ =\ v\star (u\star g)\ =\ ((u\star g)\star v)^{-1}$.
\end{center}
If $r\ =\ 1$, we are done. Suppose that $r\geq 2$. Then $T_{r}(G)$ is normal subgroup of $\hat{G}$. Further by defining condition (5),    $^{g^{-1}}v\ =\ v'$ remains a generating element of $T_{r}(G)$ and $^{g}v'\ =\ v$.  By the defining condition (4),
\begin{center}
$((u\star g)\star ^{g}v')((g\star v')\star ^{v'}u)((v'\star u)\star ^{u}g)\ =\ 1$.
\end{center}
Thus,
\begin{center}
$((u\star g)\star v)^{-1}\ =\ ((g\star v')\star ^{v'}u)((v'\star u)\star ^{u}g)$.
\end{center}

Now, $((v'\star u)\star ^{u}g)\ =\ ^{u}(^{u^{-1}}(v'\star u)\star g)$. By the induction hypothesis $(v'\star u)$ can be expressed as product generating elements of $T_{n-1}(G)$ and since $T_{n-1}(G)$ is normal, it follows that  $((v'\star u)\star ^{u}g)$ belongs to $T_{n}(G)$. Next, consider the element $((g\star v')\star ^{v'}u)$. As above 
\begin{center}
$((g\star v')\star ^{v'}u)\ =\ ^{v'}(^{v'^{-1}}(g\star v')\star u)$.
\end{center}

Clearly, $(g\star v')\ =\ (v'\star g)^{-1}$ is a generating element of $T_{r+1}(G)$ and since $T_{i}(G)$ is normal in $\hat{G}$, $^{v'^{-1}}(g\star v')$ is product of generating elements of $T_{r+1}(G)$. Since $u$ is a generating element of $T_{s-1}(G)$, $(r+1 + s-1)\ =\ n$, using second induction on $s$ in the pairs $(r, s)$ with $r+s\ =\ n$, we see that $(^{v'^{-1}}(g\star v')\star u)$ is a member of $T_{n}(G)$. This completes the proof of the proposition. 
\end{proof}

Note that the proof does not imply that $\hat{B_{n}}(G)\subseteq \hat{T_{n}}(G)$. Our next aim is to understand, to some extent, the group  structure of $SLF(G)$ and of $LF(G)$.

Let $H$ be a normal subgroup of a group $G$. It can be shown (using (5)) that $LF(H)$ is a normal subgroup of $LF(G)$ which is a subalgebra, however, it need not be an ideal of $LF(G)$. Evidently, the functor $LF$ from $\cat{GP}$ to $\cat{ML}$ preserves epimorphism.  In particular, given a normal subgroup $H$ of $G$, $LF(\nu)$ is  a surjective Lie algebra homomorphism from $LF(G)$ to $LF(G/H)$, where $\nu$ is the quotient map from $G$ to $G/H$. We try to describe its kernel. Let $T_{n}(H, G)$ denote the subgroup of $LF(G)$ generated by the set $t_{n}(H, G)\ =\ \{t_{n}(h_{1}, g_{2}, g_{3}, \cdots , g_{n})\mid h_{1}\in H, g_{i}\in G\}$. We can imitate the proof of Proposition \ref{s3p2} to establish that $T_{n} (H, G)$ is normal subgroup of $LF(G)$ for $n\geq 2$. Further, it can also be shown on similar lines that $\beta_{n}(H, G)\ =\ \{\beta_{n} (x_{1}, x_{2}, \cdots , x_{n})\mid x_{i}\in H \text{ for some } i\}\subseteq  T_{n}(H, G)$. Let $LF(H, G)$ denote the subgroup $\prod_{n=1}^{\infty}T_{n}(H, G)$ generated by the set of elements of the types $\beta_{n} (x_{1}, x_{2}, \cdots , x_{n})$, $x_{i}\in H$  for some $i$. Evidently, $LF(G, G)\ =\ LF(G)$.

\begin{prp}\label{s3p4}
The set $LF(H, G)$ defined above is an ideal of $LF(G)$ containing $LF(H)$ which is the kernel of $LF(\nu )$. More explicitly, we have a natural short exact sequence
\begin{center}
$1\longrightarrow LF(H, G)\stackrel{i}{\rightarrow}LF(G)\stackrel{\nu}{\rightarrow}LF(G/H)\longrightarrow 1$
\end{center}
of multiplicative Lie algebra.
\end{prp}

\begin{proof}
 Using the fact that $1\star x\ =\ 1\ =\ x\star 1$ for all $x$ (Proposition \ref{s3p2}(i)), we see that $\nu o i\ =\ 0$. Next the image of inclusion group homomorphism $j\ =\ H\stackrel{i}{\rightarrow}LF(G)$ is clearly contained in $LF(H, G)$.  It induces the group homomorphism $\overline{j}$ from  $G/ H$ to $LF(G)/LF(H, G)$. Clearly $\nu \circ LF(\overline{j} )\ =\ I_{LF(G/H)}$. The result follows. 
\end{proof}

Restricting the above short exact sequence to $LF(H, G)\star LF(H, G)\ =\ \prod_{n=2}^{\infty}T_{n}(H, G)\subseteq LF(G)\star LF(G)$,  the following corollary becomes immediate.

\begin{cor}\label{s3p4c1}
Given a pair $(G, H)$, where $H$ is a normal subgroup of $G$, we have a natural short exact sequence of miltiplicative Lie algebras:
\begin{center}
$1\longrightarrow \prod_{n=2}^{\infty}T_{n}(H, G)\stackrel{i}{\rightarrow}LF(H, G)\star LF(H, G)\stackrel{\nu}{\rightarrow}LF(G/H)\star LF(G/H)\longrightarrow 1$. 
\end{center}
\end{cor}

\begin{rmk}
The above results may be effectively used to relate co-homology of groups with that of multiplicative Lie algebras.
\end{rmk}

\begin{cor}\label{s3p4c2}
 $K_{n}(G)\ =\ \prod_{i=1}^{n}T_{n}(G)$ is a subgroup of $\hat{G}$ (not necessarily normal)  and we have an ascending chain
\begin{center}
$G\approx K_{1}(G)\subseteq K_{2}(G)\subseteq\cdots K_{n}(G)\subseteq K_{n+1}(G)\cdots$
\end{center}
of subgroups of $\hat{G}$. 
\end{cor}

\begin{cor}\label{s3p4c3}
$\hat{G}\ =\ \bigcup_{i=1}^{\infty}K_{i}(G)$. 
\end{cor}

\section{Steinberg and Chevalley Multiplicative Lie algebras}

The most illuminating and motivating example for this section is the Steinberg multiplicative Lie algebra. Let $R$ denote a ring with identity. Let $GL(n,R)$ denote the general linear group over the ring $R$, that is the group of $(n\times n)$ invertible matrices over $R$. The group $GL(n,R)$ is embedded in $GL(n+1,R)$ through the map

\[A\to \begin{pmatrix}A & 0 \\ 0 & 1 \end{pmatrix}.\]

Let $GL(R)$ denote the direct limit of the chain 

\[GL(1,R)\subseteq \cdots \subseteq GL(n,R) \subseteq GL(n+1,R) \subseteq \cdots\]

Let $E_{ij}^{\lambda}$ ($i\neq j$, $\lambda\in R$) denote the elementary matrix with $1$ on the main diagonal and $\lambda$ at $ij^{th}$ entry. Let $E(R)$ be the subgroup of $GL(R)$ generated by the elementary matrices $E_{ij}^{\lambda}$ ($i\neq j$, $\lambda\in R$). The Steinberg group $St(n,R)$ is the group generated by the symbols 

\[\{x_{ij}^{\lambda}\mid 1\leq i\neq  j\leq n\ \text{and}\ \lambda\in R\}\]

subject to the  Steinberg relations

\begin{enumerate}
	\item[(i)] $x_{ij}^{\lambda + \mu}\ =\ x_{ij}^{\lambda}x_{ij}^{\mu}$,
	\item[(ii)] $[x_{ij}^{\lambda}, x_{jl}^{\mu}]\ =\ x_{il}^{\lambda\mu},\ i\neq l$, and
	\item[(iii)] $[x_{ij}^{\lambda}, x_{kl}^{\mu}]\ =\ 1$ for $j\neq k$ and $i\neq l$.
\end{enumerate}

Let $St(R)$ denote the direct limit $\lim_{n\to \infty} St(n,R)$. It is known that the Steinberg group $St(R)$ is the universal central extension of $E(R)$ (see \cite{jm} for more and further detail). We have the following central extension 

\[1\rightarrow K_2(R)\rightarrow St(R)\rightarrow E(R)\rightarrow 1,\]

where $K_2(R)$ is the kernel of the natural map. It is also known that $K_2(R)$ is isomorphisc to the Schur multiplies of $E(R)$.



For each $n\geq 3$, define the multiplicative Lie algebra $St^{m}(n, R)$ generated by the set $\{ x_{ij}^{\lambda}\mid 1\leq i\neq  j\leq n\ \text{and}\ \lambda\in R\}$ subject to relations

\begin{enumerate}
	\item[(i)] $x_{ij}^{\lambda + \mu}\ =\ x_{ij}^{\lambda}x_{ij}^{\mu}$,
	\item[(ii)] $x_{ij}^{\lambda}\star x_{jl}^{\mu}\ =\ x_{il}^{\lambda\mu},\ i\neq l$, and
	\item[(iii)] $x_{ij}^{\lambda}\star x_{kl}^{\mu}\ =\ 1$ for $j\neq k$ and $i\neq l$.
\end{enumerate}

Evidently $St^{m}(n, R)$ is a subalgebra of $St^{m+1}(n, R)$. The direct limit $\lim_{n\rightarrow \infty}St^{m}(n, R)$ is called the Steinberg multiplicative Lie algebra and it is denoted by $St^{m}(R)$. We have obvious surjective  multiplicative Lie algebra homomorphism $\chi_{n}$ from $St^{m}(n, R )$ to $St(n, R )$ induced by the tautological map $x_{ij}^{\lambda}\mapsto x_{ij}^{\lambda}$ and whose kernel is the ideal generated by the set $\{[y, x](x\star y)\mid x, y\in St^{m}(n, R )\}$. It can be shown (\cite{bdi}) that $\chi_{n}$ has a group theoretic section. Motivated from further developments in \cite{lu}, we introduce the notion of Chevalley multiplicative Lie algebra.

Let us first recall Chevalley group associated to an irreducible root system $\Phi$ in an Euclidean space. For undefined terms and further details, we refer the reader \cite{rs} and \cite{jh}.  We shall always assume that rank $\Phi\geq 2$. Let $K$ be a field with at least 5 elements. Let $G(\Phi, K)$ denote the group generated by the set $S\ =\ \{x_{\alpha}(a)\mid \alpha\in \Phi , a\in K\}$ subject to the relations:

(i) $x_{\alpha}(a)x_{\alpha}(b)\ =\ x_{\alpha}(a + b),\ \alpha\in \Phi , a, b\in K$,

(ii) For each  $\alpha, \beta\neq -\alpha$, 

\begin{center}
$[x_{\alpha}(a), x_{\beta} (b)]\ =\ \prod_{i, j\geq 1}x_{i\alpha + j\beta }(c_{ij}a^{i}b^{j})$,
\end{center}

where $c_{ij}$ is canonically chosen set of integers depending only on a fixed ordering of the set $\{i\alpha + j\beta\mid i, j\geq 1\}$ (lexicographic order say) with $c_{11}\ =\ \nu_{\alpha \beta}\ =\ \pm (r+1)$, where $\beta -r\alpha$ is the smallest root in the $\alpha$-string through $\beta$.

(iii) $h_{\alpha}(ab)\ =\ h_{\alpha}(a)h_{\alpha }(b)$, where $h_{\alpha}(a)\ =\ x_{\alpha}(a)x_{-\alpha}(-a^{-1})x_{\alpha}(a)x_{\alpha}(-1)x_{-\alpha}(1)x_{\alpha}(-1)$.
 
It can be easily observed that the commutator relation 

\begin{center}
$[h_{\alpha}(a), x_{\beta }(b)]\ =\ x_{\beta}((a^{\prec \beta, \alpha \succ} -1)b)$
\end{center}
is derivable from (i) and (ii).

The group $G(\Phi, K)$ is termed as the universal Chevalley group associated to the irreducible root system $\Phi$ and the field $K$ . Note that the group depends (upto isomorphism) only on the isomorphism class of the root system $\Phi$ and is independent of the choice of the said orderings. The group $G(\Phi, K)/Z$ ($Z$ is the center) is termed as Chevalley group of adjoint type. 

The Steinberg group $St(\Phi , K)$ associated to $\Phi$ and $K$ is the universal central extension of the universal Chevalley group generated by the set $S$ of symbols subject to the relations (i) and (ii) only. More explicitly, we have a unique surjective homomorphism $\nu$ from $St (\Phi , K)$ to $G(\Phi , K)$ given by the tautological map $x_{\alpha}(a)\mapsto x_{\alpha}(a)$ whose kernel is a central subgroup of $St(\Phi , K)$ which is precisely the Schur multiplier of $G(\Phi , K)$. Indeed 

\begin{center}
$1\longrightarrow Ker\nu\stackrel{i}{\rightarrow}St(\Phi , K)\stackrel{\nu}{\rightarrow}G(\Phi , K)\rightarrow 1$
\end{center}

is universal central extension of $G(\Phi , K)$.

\begin{df}\label{c4d1}
The Chevalley multiplicative Lie algebra $G^{m}(\Phi, K)$ associated to the irreducible root system $\Phi$ and the field $K$ is defined to be the multiplicative Lie algebra generated by the set $\hat{S}\ =\ \{\hat{x}_{\alpha}(a)\mid \alpha\in \Phi , a\in K\}$ subject to the relations:

(i) $\hat{x}_{\alpha}(a)\hat{x}_{\alpha}(b)\ =\ \hat{x}_{\alpha}(a + b),\ \alpha\in \Phi , a, b\in K$,

(ii) For each  $\alpha, \beta\neq -\alpha$, 
\begin{center}
$\hat{x}_{\alpha}(a)\star \hat{x}_{\beta} (b)\ =\ \prod_{i, j\geq 1}\hat{x}_{i\alpha + j\beta }(c_{ij}a^{i}b^{j})$,
\end{center}
where $c_{ij}$ is canonically chosen set of integers depending only on a fixed ordering of the set $\{i\alpha + j\beta\mid i, j\geq 1\}$ (lexicographic order say) with $c_{11}\ =\ \nu_{\alpha \beta}\ =\ \pm (r+1)$, where $\beta -r\alpha$ is the smallest root in the $\alpha$-string through $\beta$.

(iii) $\hat{h}_{\alpha}(ab)\ =\ \hat{h}_{\alpha}(a)\hat{h}_{\alpha }(b)$, where $\hat{h}_{\alpha}(a)\ =\ \hat{x}_{\alpha}(a)\hat{x}_{-\alpha}(-a^{-1})\hat{x}_{\alpha}(a)\hat{x}_{\alpha}(-1)\hat{x}_{-\alpha}(1)\hat{x}_{\alpha}(-1)$.

(iv) $\hat{h}_{\alpha}(a)\star \hat{x}_{\beta} (b)\ =\ \hat{x}_{\beta}((a^{\prec \beta, \alpha \succ} -1)b)$ for all $\alpha, \beta\in \Phi$ and $a, b\in K$.
\end{df}

It follows from $(iv)$ above that $G^{m}(\Phi , K)$ is Lie perfect. Clearly, the map $\hat{x}_{\alpha}(a)\mapsto x_{\alpha}(a)$ respects the relations, and as such we have a surjective multiplicative Lie algebra homomorphism $\eta$ from $G^{m}(\Phi , K)$ to $G(\Phi , K)$ given by $\eta (\hat{x}_{\alpha}(a))\ =\ x_{\alpha}(a)$, where $G(\Phi , K)$ is with improper multiplicative Lie algebra structure. 

Similarly, we introduce the Steinberg multiplicative Lie algebra  $St^{m}(\Phi , K)$ to be the multiplicative Lie algebra generated by the set $\hat{S}$ of symbols subject to the relations (i), (ii), and (iv) above. As such $St^{m}(\Phi , K)$ is also Lie perfect. Here again, we have surjective multiplicative Lie algebra homomorphism $\chi$ from $St^{m}(\Phi , K )$ to $St(\Phi , K)$ with improper Lie structure.

\begin{prp}\label{c4p1}
\begin{enumerate}
	\item[(i)] Kernel of $\eta$ is the ideal $H$ of $G^{m}(\Phi, K)$ generated by the set $\{(x\star y)[y, x]\mid x, y\in G^{m}(\Phi , K)\}$ and which  is a central subgroup of $G^{m}(\Phi, K)$. We have a semi-central extension

\begin{center}
$1\longrightarrow H\stackrel{i}{\rightarrow}G^{m}(\Phi, K)\stackrel{\eta}{\rightarrow}G(\Phi , K)\longrightarrow 1$.
\end{center}
	\item[(ii)] Kernel of $\chi$ is the normal subgroup $H^{m}$ of $St^{m}(\Phi, K)$ generated by the set $\{(x\star y)[y, x]\mid x, y\in St^{m}(\Phi , K)\}$ and which is central subgroup of $St^{m}(\Phi, K)$. We have a group theoretic central extension
\begin{center}
$1\longrightarrow H^{m}\stackrel{i}{\rightarrow}St^{m}(\Phi, K)\stackrel{\hat{\chi}}{\rightarrow}St(\Phi , K)\longrightarrow 1$.
\end{center}.
	\item[(iii)] $\chi$ admits a group theoretic section.
	\item[(iv)] We have another universal central extension
\begin{center}
$1\longrightarrow Ker\nu^{m}\stackrel{i}{\rightarrow}St^{m}(\Phi , K)\stackrel{\hat{\nu^{m}}}{\rightarrow}G^{m}(\Phi , K)\rightarrow 1$
\end{center}
of multiplicative Lie algebras induced by the corresponding universal central extension of groups.
  \item[(v)] $Ker\nu^{m}$ is the Schur multiplier of the multiplicative Lie algebra.
\end{enumerate}
\end{prp}

\begin{proof}
 (i) Let $H$ be the ideal of $G^{m}(\Phi , K)$ generated by the set $\{(x\star y)[y, x]\mid x, y\in G^{m}(\Phi , K)\}$. It follows from Proposition \ref{s2p1}(iv) that  $H\subseteq Ker \eta $. Thus, $\eta $ induces a Lie homomorphism $\overline{\eta}$ from $G^{m}(\Phi , K)/ H$ to the improper multiplicative Lie algebra $G(\Phi , K)$ which is given by $\overline{\eta }(xH)\ =\ \eta (x)$. The map $x_{\alpha}(a)\mapsto \hat{x}_{\alpha}(a)H$ respects the defining relation of the group $G(\Phi ,K)$ and as such induces group theoretic inverse of $\overline{\eta}$. This shows that $H$ is the kernel of $\eta$. 

 Since $St^{m}(\Phi , K)$ is Lie perfect,

\begin{align*}
[(x\star y)[y, x], (u\star v)] & = (x\star y)[y, x] (u\star v)[y, x]^{-1}(x\star y)^{-1}(u\star v)^{-1} \\
                               & = (x\star y)(y\star x) (u\star v)(y\star x)^{-1}(x\star y)^{-1}(u\star v)^{-1} \text{ (by Proposition \ref{s2p1}(iii))} \\
															 & = 1
\end{align*}

for all $x, y, u, v\in G^{m}(\Phi , K)$. Further, by Proposition \ref{s2p1}(iv),

\begin{center}
$[[x,y](y\star x)\star z, u\star v]\ =\ [[x,y](y\star x), z]\star (u\star v)\ =\ 1\star (u\star v)\ =\ 1$
\end{center}

Thus, $\eta $ induces a Lie homomorphism $\overline{\eta}$ from $G^{m}(\Phi , K)/ H$ to the improper multiplicative Lie algebra $G(\Phi , K)$ which is given by $\overline{\eta }(xH)\ =\ \eta (x)$. The map $x_{\alpha}(a)\mapsto \hat{x}_{\alpha}(a)H$ respects the defining relation of the group $G(\Phi ,K)$ and as such induces group theoretic inverse of $\overline{\eta}$. This shows that $H$ is the kernel of $\eta$.

\vspace{0.2 cm}

(ii) The proof is on the same lines as that of (i).

\vspace{0.2 cm}

(iii) Since any group theoretic central extension of $St(\Phi , K)$ splits (\cite{bdi}), the result follows from (ii).  

\vspace{0.2 cm}

(iv) Follows from the functoriality.

\vspace{0.2 cm}

(v) Follows from (iv). 
\end{proof}

\begin{rmk}
Not every central extension of a Chevalley group may split, for example

\[1\rightarrow Z(GL(2,,R))\rightarrow GL(2,R)\rightarrow PGL(2,R)\rightarrow 1.\]

However, there are Chevalley groups with trivial Schur multipliers with the property that all of their central extension splis (see \cite[Corollary 2, p. 52]{rs}). For such Chevalley groups, $\eta$ has a group theoretic section.
\end{rmk}

\end{document}